\documentclass[11pt,reqno]{amsart}
\usepackage{a4wide}
\usepackage[utf8x]{inputenc}
\usepackage[T1]{fontenc}
\usepackage[english]{babel}
\usepackage{hyperref}
\usepackage{amsfonts,amsmath,amssymb,amsthm,amsrefs}
\usepackage{latexsym}
\usepackage{layout}
\usepackage{dsfont}
\usepackage{xcolor}

\usepackage{framed}

\usepackage{graphicx}
\newtheorem{theorem}{Theorem}[section]
\newtheorem{proposition}{Proposition}[section]
\newtheorem{lemma}{Lemma}[section]

\newtheorem{remark}{Remark}[section]
\newtheorem*{proposition*}{Proposition}
\newtheorem{defi}[theorem]{Definition}

\newcommand{\p}{\mathbb{P}}
\newcommand{\e}{\mathbb{E}}
\newcommand{\pK}{\mathbb{P}^{(K)}}
\newcommand{\eK}{\mathbb{E}^{(K)}}

\newcommand{\reals}{\mathbb{R}}
\newcommand{\ind}{\mathbf{1}}

\newcommand{\ruintime}{\sigma^\pi}
\newcommand{\Hq}{H^{(q)}_K}
\newcommand{\Hqprime}{H^{(q)\prime}_K}

\allowdisplaybreaks

\begin{document}

\title[A control problem with linearly bounded control rates]{A stochastic control problem with linearly bounded control rates in a Brownian model}

\author[]{Jean-Fran\c{c}ois Renaud}
\address{D\'epartement de math\'ematiques, Universit\'e du Qu\'ebec \`a Montr\'eal (UQAM), 201 av.\ Pr\'esident-Kennedy, Montr\'eal (Qu\'ebec) H2X 3Y7, Canada}
\email{renaud.jf@uqam.ca}

\author[]{Clarence Simard}
\address{D\'epartement de math\'ematiques, Universit\'e du Qu\'ebec \`a Montr\'eal (UQAM), 201 av.\ Pr\'esident-Kennedy, Montr\'eal (Qu\'ebec) H2X 3Y7, Canada}
\email{simard.clarence@uqam.ca}


\date{\today}


\thanks{Funding in support of this work was provided by the Natural Sciences and Engineering Research Council of Canada (NSERC)}

\begin{abstract}
Aiming for more realistic optimal dividend policies, we consider a stochastic control problem with linearly bounded control rates using a performance function given by the expected present value of dividend payments made up to ruin. In a Brownian model, we prove the optimality of a member of a new family of control strategies called delayed linear control strategies, for which the controlled process is a refracted diffusion process. For some parameters specifications, we retrieve the strategy initially proposed in \cite{avanzi-wong_2012} to regularize dividend payments, which is more consistent with actual practice.
\end{abstract}

\maketitle

%

\section{Introduction and main result}

In this paper, we consider an optimal stochastic control problem with absolutely continuous and linearly bounded control strategies in a Brownian model. Our problem is a version of de Finetti's dividend problem in which payment/control strategies can be at most a (fixed) fraction of the current wealth/state. For this problem, an optimal strategy is formed by a delayed linear control strategy. Loosely speaking, for such a strategy, no dividends are paid below a barrier and level-dependent dividend payments are made when the process is above that barrier, allowing for more regular dividend payments over time. As the underlying state process is a Brownian motion, in order to solve this stochastic control problem, we study a specific \textit{refracted} diffusion process, i.e., a dynamic alternating between a Brownian motion with drift and an Ornstein-Uhlenbeck process.

\subsection{Model and problem formulation}

On a filtered probability space $\left( \Omega, \mathcal{F}, \left\lbrace \mathcal{F}_t, t \geq 0\right\rbrace, \p \right)$, let $X=\left\lbrace X_t , t \geq 0 \right\rbrace$ be a Brownian motion with drift, i.e.
\begin{equation}\label{eq:BM}
\mathrm{d}X_t = \mu \mathrm{d}t + \sigma \mathrm{d}B_t ,
\end{equation}
where $\mu \in \reals$ and $\sigma > 0$, and where $B=\left\lbrace B_t , t \geq 0 \right\rbrace$ is a standard Brownian motion. Using the language of ruin theory, $X$ is the (uncontrolled) surplus process.

A dividend strategy $\pi$ is represented by a non-decreasing, left-continuous and adapted stochastic process $L^\pi = \left\lbrace L^\pi_t , t \geq 0 \right\rbrace$, where $L^\pi_t$ represents the cumulative amount of dividends paid up to time $t$ under this strategy. For a given strategy $\pi$, the corresponding controlled surplus process $U^\pi = \left\lbrace U^\pi_t , t \geq 0 \right\rbrace$ is defined by $U^\pi_t = X_t - L^\pi_t$.

While de Finetti's classical dividend problem is a singular stochastic control problem, we are interested in an adaptation where the admissible strategies are absolutely continuous with a linearly bounded rate. 
\begin{defi}[Admissible strategies]
Fix a constant $K > 0$. A strategy $\pi$ is said to be admissible if it is absolutely continuous and linearly bounded, i.e., $\mathrm{d}L_t^\pi = \ell^\pi_t \mathrm{d}t$, with $0 \leq \ell^\pi_t \leq K U^\pi_t$, for all $t>0$. The set of admissible dividend strategies will be denoted by $\Pi^K$.
\end{defi}
Thus, for an admissible strategy $\pi$, we have
$$
\mathrm{d}U^\pi_t = \mathrm{d}X_t - \ell^\pi_t \mathrm{d}t = \left( \mu - \ell^\pi_t \right) \mathrm{d}t + \sigma \mathrm{d}B_t .
$$

For a given discounting rate $q \geq 0$, the value/performance function associated to an admissible strategy $\pi$ and initial value $U_0=X_0=x$ is given by
$$
v_\pi (x) = \e_x \left[ \int_0^{\ruintime} \mathrm{e}^{-q t} \ell^\pi_t \mathrm{d}t \right]
$$
where $\ruintime=\inf \left\lbrace t > 0 \colon U_t^\pi < 0 \right\rbrace$. Note that $v_\pi(x)=0$ for all $x \leq 0$. For notational simplicity, but without loss of generality, we have chosen the ruin level to be $0$, as it is often the case in ruin theory when dealing with space-homogeneous surplus processes.

The goal of this stochastic control problem is to find the optimal value function $v_\ast$ defined by
$$
v_\ast (x) = \sup_{\pi \in \Pi^K} v_\pi (x)
$$
and, if it exists, an optimal strategy $\pi_\ast \in \Pi^K$ such that
$$
v_{\pi_\ast} (x) = v_\ast (x) ,
$$
for all $x>0$.

\subsection{Main result}

As alluded to above, the optimal control strategy is of bang-bang type, i.e., dividends are either paid out at the maximal surplus-dependent rate or nothing is paid, depending on the value of the controlled process. Mathematically, let us define the following \textit{refracted} diffusion process: for fixed $b, K \geq 0$, let $U^b=\left\lbrace U^b_t , t \geq 0 \right\rbrace$ be given by
\begin{equation}\label{eq:main-sde}
\mathrm{d}U^b_t = \left( \mu - K U^b_t \ind_{\{U^b_t > b\}} \right) \mathrm{d}t + \sigma \mathrm{d}B_t .
\end{equation}
We prove the existence of a strong solution to this stochastic differential equation in Appendix~\ref{an:existence_proof}. This process is the controlled process associated to the (admissible) control strategy $\ell^b_t = K U^b_t \ind_{\{U^b_t > b\}}$, which consists of paying out at the maximal surplus-dependent rate $K$ when the controlled process is above level $b$ and pay nothing below $b$. This control strategy will be denoted by $\pi_b$ and called a delayed linear control strategy at level $b$ with rate $K$, when $b>0$, and simply a linear control strategy with rate $K$, when $b=0$. In both cases, it is an admissible strategy, i.e., $\pi_b \in \Pi^K$ for any $b \geq 0$.

Define
\begin{equation}\label{eq:Delta}
\Delta = - \Hq(0)/\Hqprime(0) ,
\end{equation}
where the function $\Hq$ is defined in~\eqref{eq:H-scale-function}. Note that $\Delta$ depends on both the control problem parameters $q$ and $K$ and the model parameters $\mu$ and $\sigma$. Define also $b^\ast$ as the (unique) root of Equation~\eqref{eq:eqn-optimal-barrier}; see below.

Here is an explicit solution to the control problem:
\begin{theorem}\label{T:main}
Fix $q > 0$ and $K > 0$. The following hold:
\begin{itemize}
\item[(i)] If $\mu K/q^2 \leq \Delta$, then the linear control strategy with rate $K$ is optimal.
\item[(ii)] If $\mu K/q^2 > \Delta$, then the delayed linear control strategy at level $b^\ast>0$ with rate $K$ is optimal.
\end{itemize}
\end{theorem}

\subsection{Related literature}\label{sec:related_lit}

If a delayed linear control strategy at level $b>0$ with rate $K$ is implemented, then dividends are paid continuously only when the process is above level $b$, according to an Ornstein-Uhlenbeck process with mean-reverting level $\mu/K$. In other words, when $b>0$, we have to wait for sufficient capital before paying out dividends according to this linear strategy. If the linear strategy with rate $K$ is implemented, i.e., when $b=0$, then dividends are paid without interruption up to the time of ruin, according to an Ornstein-Uhlenbeck process with mean-reverting level $\mu/K$. This last strategy is the one initally proposed in \cite{avanzi-wong_2012} to regularize dividend payments; it is also studied in \cite{albrecher-cani_2017} for a compound Poisson process with drift. In Section~6 of \cite{avanzi-wong_2012}, the authors hinted at delayed linear control strategies. In some sense, delaying dividend payments brings \textit{additional safety} by avoiding payments when the process is close to the ruin level. Theorem~\ref{T:main} tells us that, for our control problem, either the optimal strategy is a linear control strategy, as studied in \cite{avanzi-wong_2012}, or it is a member of the family of delayed linear control strategies. In the latter case, Proposition~\ref{P:root} tells us how to find the optimal level $b^\ast$.

Note that our problem has similarities with another variation on de Finetti's optimal dividends problem, in which admissible strategies have a bounded rate, i.e., $0 \leq \ell^\pi_t \leq \widetilde{K}$, for all $t>0$, for a fixed constant $\widetilde{K}$; see, e.g., \cite{jeanblanc-shiryaev_1995}. However, for that problem, the optimal control process is a two-valued Brownian motion with drift, also called a refracted Brownian motion with drift. In the classical de Finetti's control problem, admissible strategies are not necessarily absolutely continuous nor bounded. For the classical problem, the optimal control process is a reflected Brownian motion with drift and the optimal reflection level is known to be
\begin{equation}\label{eq:classical-optimal-level}
c^\ast = \frac{\sigma^2}{\sqrt{\mu^2 + 2q\sigma^2}} \ln \left( \frac{\mu + \sqrt{\mu^2 + 2q\sigma^2}}{-\mu + \sqrt{\mu^2 + 2q\sigma^2}} \right) .
\end{equation}
Intuitively, in our problem, if $K$ becomes large, then we should get closer to the classical problem.

The rest of the paper is organized as follows. In Section 2, we provide background material on Ornstein-Uhlenbeck diffusion processes, we compute the value function of an arbitrary (delayed) linear control strategy and then we find the (candidate) optimal barrier level $b^\ast$. In Section 3, we give a verification lemma for this control problem and prove Theorem~\ref{T:main}. Finally, Section~\ref{S:discussion} aims at giving a second look at the main results with a short discussion and numerical illustrations. More technical material is provided in the appendices, including a proof for the existence of a refracted Ornstein-Uhlenbeck diffusion process and ordinary differential equations related to the control problem.

\section{Delayed linear control strategies}

Before computing the value function of an arbitrary linear control strategy and then finding the optimal barrier level, let us recall some results about Ornstein-Uhlenbeck diffusion processes.

\subsection{Preliminaries on Ornstein-Uhlenbeck diffusion processes}

First, let us recall some results on first-passage problems for the Brownian motion with drift $X$, given in~\eqref{eq:BM}, and the Ornstein-Uhlenbeck diffusion process $U=\left\lbrace U_t , t \geq 0 \right\rbrace$ given by
\begin{equation}\label{eq:OU}
\mathrm{d}U_t = \left( \mu - K U_t \right) \mathrm{d}t + \sigma \mathrm{d}B_t ,
\end{equation}
where $K \geq 0$. Clearly, if $K=0$, then $U=X$. In what follows, we will use the following notation: the law of $U$ when starting from $U_0 = x$ is denoted by $\pK_x$ and the corresponding expected value operator by $\eK_x$. Define the following first-passage stopping times: for $a \in \reals$, let $\tau_a = \inf \left\lbrace t \colon U_t = a \right\rbrace$.

In what follows, we assume $q>0$. First, we consider the case $K=0$, i.e.\ the case when $U=X$ is a Brownian motion with drift. It is known (see, e.g., \cite{kyprianou_2014}) that, for $0 \leq x \leq a$,
\begin{equation}\label{eq:two-sided-BM}
\e^{(0)}_x \left[ \mathrm{e}^{-q \tau_a} \ind_{\{\tau_a < \tau_0\}} \right] = \frac{W^{(q)}(x)}{W^{(q)}(a)} ,
\end{equation}
where, for $x>0$,
\begin{equation}\label{eq:q-scale-function}
W^{(q)}(x)= \frac{2}{\sqrt{\mu^2+2q\sigma^2}} \mathrm{e}^{-(\mu/\sigma^2) x} \sinh \left( (1/\sigma^2) \sqrt{\mu^2+2q\sigma^2} x \right)
\end{equation}
and, for $x \leq 0$, $W^{(q)}(x)=0$. It is known that $W^{(q) \prime}$ is convex.

Now, we consider the case $K>0$, i.e., the case when $U$ is an Ornstein-Uhlenbeck diffusion process. It is known (see, e.g., \cite{alili-et-al_2005} and references therein) that $U$ is a recurrent diffusion, so $\tau_a$ is finite, almost surely. If $\mu=0$ and $\sigma=1$, then, for $x<a$,
$$
\eK_x \left[ \mathrm{e}^{-q \tau_a} \ind_{\{\tau_a < \infty\}} \right] = \frac{\mathrm{e}^{K x^2/2} \mathrm{D}_{-q/K} \left(-x \sqrt{2K} \right)}{\mathrm{e}^{K a^2/2} \mathrm{D}_{-q/K} \left(-a \sqrt{2K} \right)} ,
$$
where $\mathrm{D}_{-\lambda}$ is the parabolic cylinder function defined, for $x \in \reals$, by
$$
\mathrm{D}_{-\lambda}(x) = \frac{1}{\Gamma(\lambda)} \mathrm{e}^{-x^2/4} \int_0^\infty t^{\lambda-1} \mathrm{e}^{-xt-t^2/2} \mathrm{d}t .
$$
We see that $x \mapsto \mathrm{e}^{K x^2/2} \mathrm{D}_{-q/K} \left(-x \sqrt{2K} \right)$ is increasing and convex.

Consequently, for $\mu \in \reals$ and $\sigma>0$, we can deduce that, for $x \geq a$,
\begin{equation}\label{eq:first-passage-OU}
\eK_x \left[ \mathrm{e}^{-q \tau_a} \ind_{\{\tau_a < \infty\}} \right] = \frac{\Hq \left( x \right)}{\Hq \left( a \right)} ,
%
\end{equation}
where
\begin{equation}\label{eq:H-scale-function}
\Hq (x) := \mathrm{e}^{K \left(x-\mu/K \right)^2/2\sigma^2} \mathrm{D}_{-q/K} \left(\left(\frac{x-\mu/K}{\sigma} \right) \sqrt{2K} \right) .
\end{equation}
It is easy to verify that $\Hq$ is convex.
%

For more analytical properties of $W^{(q)}$ and $\Hq$, see Appendix~\ref{an:odes}.

\subsection{Value function and optimal barrier level}

Fix $q > 0$ and $K>0$. Recall that, for an arbitrary $b \geq 0$, the process $U^b=\left\lbrace U^b_t , t \geq 0 \right\rbrace$ is the controlled process associated to the delayed linear strategy $\pi_b \in \Pi^K$. This refracted diffusion process is the solution to the SDE in~\eqref{eq:main-sde}, which we recall here:
$$
\mathrm{d}U^b_t = \left( \mu - K U^b_t \ind_{\{U^b_t > b\}} \right) \mathrm{d}t + \sigma \mathrm{d}B_t .
$$
Denote the ruin time of this process by $\sigma^b=\inf \left\lbrace t > 0 \colon U_t^b < 0 \right\rbrace$. Note that, if $b=0$, then $U^0=\left\lbrace U^0_t , 0 \leq t \leq \sigma^0 \right\rbrace$ has the same law as $U=\left\lbrace U_t , 0 \leq t \leq \tau_0 \right\rbrace$, where $U$ is the Ornstein-Uhlenbeck process defined in~\eqref{eq:OU}.

%

Denote the value function associated to $\pi_b$ by
$$
v_b (x) = \e_x \left[ \int_0^{\sigma^b} \mathrm{e}^{-q t} \left( K U^b_t \ind_{\{U^b_t > b\}} \right) \mathrm{d}t \right] , \quad x \geq 0 .
$$

%

\begin{proposition}\label{P:value-linear}
For $b \geq 0$, we have
$$
v_b (x) =
\begin{cases}
\frac{K}{q+K} C_b W^{(q)}(x) & \text{if $0 \leq x \leq b$,}\\
\frac{K}{q+K} \left[ x + \frac{\mu}{q} + D_b \Hq(x) \right] & \text{if $x > b$,}
\end{cases}
$$
where
$$
C_b = \frac{\Hq(b) - \left( b + \frac{\mu}{q} \right) \Hqprime(b)}{W^{(q) \prime} (b) \Hq(b) - W^{(q)}(b) \Hqprime(b)} \qquad \text{and} \qquad D_b = \frac{W^{(q)}(b) - \left( b + \frac{\mu}{q} \right) W^{(q) \prime} (b)}{W^{(q) \prime} (b) \Hq(b) - W^{(q)}(b) \Hqprime(b)} .
$$
\end{proposition}
\begin{proof}
Fix $b>0$. First assume that $0<x\leq b$. By the strong Markov property, we have
$$
v_b(x) = \e^{(0)}_x \left[ \mathrm{e}^{-q \tau_b} \ind_{\{\tau_b < \tau_0\}} \right] v_b(b) = \frac{W^{(q)}(x)}{W^{(q)}(b)} v_b(b) ,
$$
where we used the identity in~\eqref{eq:two-sided-BM}. Now, if we assume that $x > b$, then by the strong Markov property we have
$$
v_b(x) = K \eK_x \left[ \int_0^{\tau_b} \mathrm{e}^{-q t} U_t \mathrm{d}t \right] + \eK_x \left[ \mathrm{e}^{-q \tau_b} \ind_{\{\tau_b < \infty\}} \right] v_b(b) ,
$$
where, using the identity in~\eqref{eq:first-passage-OU} and using the strong Markov property one last time (as in \cite{avanzi-wong_2012}),
$$
\eK_x \left[ \int_0^{\tau_b} \mathrm{e}^{-q t} U_t \mathrm{d}t \right] = \eK_x \left[ \int_0^\infty \mathrm{e}^{-q t} U_t \mathrm{d}t \right] - \eK_x \left[ \mathrm{e}^{-q \tau_b} \ind_{\{\tau_b < \infty\}} \right] \eK_b \left[ \int_0^\infty \mathrm{e}^{-q t} U_t \mathrm{d}t \right] .
$$
It is well known that, for $t \geq 0$,
$$
\eK_x \left[ U_t \right] = x \mathrm{e}^{-Kt} + \frac{\mu}{K} \left(1 - \mathrm{e}^{-Kt} \right) ,
$$
so we have
$$
\eK_x \left[ \int_0^\infty \mathrm{e}^{-q t} U_t \mathrm{d}t \right] = \frac{1}{q+K} \left( x + \frac{\mu}{q} \right) .
$$

Consequently
$$
v_b(x) =
\begin{cases}
\frac{W^{(q)}(x)}{W^{(q)}(b)} v_b(b) & \text{if $0 \leq x \leq b$,}\\
\frac{K}{q+K} \left( x + \frac{\mu}{q} \right) + \left[ v_b (b) - \frac{K}{q+K} \left( b + \frac{\mu}{q} \right) \right] \frac{\Hq (x)}{\Hq (b)} & \text{if $x > b$.}
\end{cases}
$$

To compute $v_b(b)$, we will use approximations of the delayed linear strategy at level $b$. Fix $n \geq 1$ and let us implement the following strategy: when the (controlled) process reaches $b$, then dividends are paid continuously at rate $K$ until the controlled process goes back to $b-1/n$; dividend payments resume when the controlled process reaches $b$ again. Denote this strategy by $\overline{\pi}_b^n$ and its value function by $\overline{v}_b^n$. Clearly, $v_b(b) \leq \overline{v}_b^n (b)$ and thus $v_b(b) \leq \lim_{n \to \infty} \overline{v}_b^n(b)$.

As for $v_b$ above, we have the following decompositions:
$$
\overline{v}_b^n(b-1/n) = \e^{(0)}_{b-1/n} \left[ \mathrm{e}^{-q \tau_b} \ind_{\{\tau_b < \tau_0\}} \right] \overline{v}_b^n(b) ,
$$
and
$$
\overline{v}_b^n(b) = \eK_b \left[ \int_0^{\tau_{b-1/n}} \mathrm{e}^{-q t} K U_t \mathrm{d}t \right] + \eK_b \left[ \mathrm{e}^{-q \tau_{b-1/n}} \ind_{\{\tau_{b-1/n} < \infty\}} \right] \overline{v}_b^n(b-1/n) .
$$
Solving for $\overline{v}_b^n(b)$, we get
$$
\overline{v}_b^n(b) = \frac{\eK_b \left[ \int_0^{\tau_{b-1/n}} \mathrm{e}^{-q t} K U_t \mathrm{d}t \right]}{1 - \e^{(0)}_{b-1/n} \left[ \mathrm{e}^{-q \tau_b} \ind_{\{\tau_b < \tau_0\}} \right] \eK_b \left[ \mathrm{e}^{-q \tau_{b-1/n}} \ind_{\{\tau_{b-1/n} < \infty\}} \right]} .
$$
Since
\begin{align*}
\eK_b \left[ \int_0^{\tau_{b-1/n}} \mathrm{e}^{-q t} K U_t \mathrm{d}t \right] &= \frac{K}{q+K} \left( b + \frac{\mu}{q} \right) - \frac{K}{q+K} \left( (b-1/n) + \frac{\mu}{q} \right) \frac{\Hq (b)}{\Hq \left( b-1/n \right)} ,\\
\e^{(0)}_{b-1/n} \left[ \mathrm{e}^{-q \tau_b} \ind_{\{\tau_b < \tau_0\}} \right] &= \frac{W^{(q)}(b-1/n)}{W^{(q)}(b)} ,\\
\eK_b \left[ \mathrm{e}^{-q \tau_{b-1/n}} \ind_{\{\tau_{b-1/n} < \infty\}} \right] &= \frac{\Hq (b)}{\Hq \left( b-1/n \right)} ,
\end{align*}
dividing by $1/n$ and taking the limit when $n \to \infty$, we get
$$
n \eK_b \left[ \int_0^{\tau_{b-1/n}} \mathrm{e}^{-q t} K U_t \mathrm{d}t \right] \longrightarrow_{n \to \infty} \frac{K}{q+K} - \frac{K}{q+K} \left( b + \frac{\mu}{q} \right) \frac{\Hqprime (b)}{\Hq (b)}
$$
and
$$
n \left\lbrace 1 - \e^{(0)}_{b-1/n} \left[ \mathrm{e}^{-q \tau_b} \ind_{\{\tau_b < \tau_0\}} \right] \eK_b \left[ \mathrm{e}^{-q \tau_{b-1/n}} \ind_{\{\tau_{b-1/n} < \infty\}} \right] \right\rbrace \longrightarrow_{n \to \infty} \frac{W^{(q)\prime}(b)}{W^{(q)}(b)} - \frac{\Hqprime (b)}{\Hq (b)} .
$$

Putting the pieces together, we finally get
$$
\lim_{n \to \infty} \overline{v}_b^n(b) = \frac{\frac{K}{q+K} - \frac{K}{q+K} \left( b + \frac{\mu}{q} \right) \frac{\Hqprime \left( b \right)}{\Hq \left( b \right)}}{\frac{W^{(q)\prime}(b)}{W^{(q)}(b)} - \frac{\Hqprime \left( b \right)}{\Hq \left( b \right)}} .
$$

Similarly, for a fixed $n \geq 1$, we can implement the following strategy: wait until the (controlled) process reaches $b+1/n$ before paying dividends at rate $K$ and do so until the controlled process goes back to $b$; dividend payments resume when the controlled process reaches $b+1/n$ again. Denote this strategy by $\underline{\pi}_b^n$ and its value function by $\underline{v}_b^n$. Clearly, $v_b(b) \geq \underline{v}_b^n (b)$ and thus $v_b(b) \geq \lim_{n \to \infty} \underline{v}_b^n(b)$.

It easy to verify that $\lim_{n \to \infty} \underline{v}_b^n(b) = \lim_{n \to \infty} \overline{v}_b^n(b)$. The result follows after algebraic manipulations.
\end{proof}

\begin{remark}
Note that, if $b=0$, then the expression just obtained for $v_0$ is in principle the same as the one obtained in \cite{avanzi-wong_2012}.

Note also that, if $b>0$, then a direct computation shows that $v_b(b-)=v_b(b+)$ and $v^\prime_b(b-)=v^\prime_b(b+)$. In other words, the continuous and the smooth pasting conditions are verified.
\end{remark}

\begin{remark}
The function $x \mapsto W^{(q)}(x)/W^{(q) \prime}(b)$ plays a fundamental role in de Finetti's classical problem. It is the value function, when starting from level $x$, of a barrier control strategy at level $b$. As mentioned above, in that problem, the optimal barrier level is given by $b=c^\ast$ as defined in~\eqref{eq:classical-optimal-level}.
\end{remark}

We now want to find the best delayed linear strategy within the sub-class of delayed linear strategies $\left\lbrace \pi_b , b \geq 0 \right\rbrace$, i.e., we want to find $b^\ast$ such that $\pi_{b^\ast}$ outperforms any other $\pi_b$. For analytical reasons, namely to be able to apply Ito's formula, the optimal barrier level $b=b^\ast$ should be such that $v^{\prime \prime}_b(b-)=v^{\prime \prime}_b(b+)$. In this direction, using identities from Appendix~\ref{an:odes}, let us define our candidate for the optimal barrier level $b^\ast$: it is the value of $b$ such that
\begin{equation}\label{eq:eqn-optimal-barrier}
\frac{K}{q} = \frac{\frac{\Hq (b)}{\Hqprime (b)} - \frac{W^{(q)} (b)}{W^{(q) \prime} (b)}}{\frac{W^{(q)} (b)}{W^{(q) \prime} (b)} - \left(b + \frac{\mu}{q} \right)} .
\end{equation}

\begin{proposition}\label{P:root}
If $\mu K/q^2 > \Delta$, then there exists a unique solution $b^\ast \in (0,c^\ast)$ to Equation~\eqref{eq:eqn-optimal-barrier}.
\end{proposition}
\begin{proof}
First, note that Equation~\eqref{eq:eqn-optimal-barrier} is equivalent to
\begin{equation}\label{eq:optimal-barrier-2}
\left(K+q \right) \left[ \frac{W^{(q)} (b)}{W^{(q) \prime} (b)} - \frac{\mu}{q} \right] + K \left( \frac{\mu}{K} - b \right) = q \frac{\Hq (b)}{\Hqprime (b)} .
\end{equation}
Set $h(b):= f(b)+g(b)$, with
\begin{align*}
f(b) &:= \left(K+q \right) \left[ \frac{W^{(q)} (b)}{W^{(q) \prime} (b)} - \frac{\mu}{q} \right] ,\\
g(b) &:= K \left( \frac{\mu}{K} - b \right) .
\end{align*}
Elementary algebraic manipulations lead to
\[
\frac{W^{(q)}(b)}{W^{ (q) \prime}(b)} = \frac{\sigma^2}{\sqrt{\mu^2+2q\sigma^2}\coth\left(b\frac{\sqrt{\mu^2+2q\sigma^2}}{\sigma^2} \right)-\mu} .
\]
From the properties of the hyperbolic cotangent function, we deduce that $W^{(q)}(b)/W^{(q) \prime} (b)$ is increasing on $[0,\infty)$. Note also that $W^{(q)}(c^\ast)/W^{(q) \prime} (c^\ast) = \mu/q$, where $c^\ast$ is given in~\eqref{eq:classical-optimal-level}. In other words, $f$ is an increasing function crossing zero at $b=c^\ast$, while $g$ is a decreasing function crossing zero at $b=\mu/K$. Finally, we see that $h(0)=-\mu K/q<0$.

On the other hand, we can verify that, for all $b \geq 0$,
\begin{equation}\label{eq:inequality}
q \frac{\Hq (b)}{\Hqprime (b)} < g(b) ,
\end{equation}
thanks to the following identities: $H_K^{(q)\prime} (x) = - \frac{q \sqrt{2K}}{K\sigma} H_{K}^{(q+K)} (x)$ and
$$
H_K^{(q)} (x) = \left( \frac{q}{K} + 1 \right) H_{K}^{(q+2K)} (x) + d(x) H_{K}^{(q+K)} (x) ,
$$
where $d(x)=\frac{\sqrt{2K}}{\sigma} \left(x- \frac{\mu}{K} \right)$. The second identity can be verified easily using integration by parts.

Under our assumptions, $h(0)<q \Hq (0)/\Hqprime (0)$. By the intermediate value theorem, since $f(c^\ast)=0$, together with the inequality in~\eqref{eq:inequality}, we can deduce that there exists a solution $b \in (0,c^\ast)$ to Equation~\eqref{eq:eqn-optimal-barrier}.

Assume there exists two solutions $0<b_1 < b_2<c^\ast$ to Equation~\eqref{eq:eqn-optimal-barrier}. Then, by the mean value theorem, there exists $a_h, a_H \in (b_1,b_2)$ such that
$$
h^\prime (a_h) = q \frac{\mathrm{d}}{\mathrm{d}b} \left(\frac{\Hq}{\Hqprime} \right) (b) \Biggr\vert_{b=a_H}
$$
which is a contradiction. Indeed, on one hand, we have
$$
h^\prime (b) = q - \left(K+q \right) \frac{W^{(q)} (b) W^{(q) \prime \prime} (b)}{\left( W^{(q) \prime} (b) \right)^2} > q
$$
since $W^{(q) \prime \prime} (b)<0$ on $(0,c^\ast)$, and on the other hand, we have
$$
\frac{\mathrm{d}}{\mathrm{d}b} \left(\frac{\Hq}{\Hqprime} \right) (b) = 1 - \frac{\Hq (b) H_K^{(q) \prime \prime} (b)}{\left( H_K^{(q) \prime}(b) \right)^2} < 1 ,
$$
for all $b$, since $\Hq$ is convex.
\end{proof}

\section{Verification Lemma and proof of Theorem~\ref{T:main}}

Here is the verification lemma of our stochastic control problem.

\begin{lemma}\label{verificationlemma}
Suppose that $\hat{\pi} \in \Pi^K$ is such that $v_{\hat{\pi}}$ is twice continuously differentiable and that, for all $x > 0$,
\begin{equation}\label{eq:HJB}
\left( \frac{\sigma^2}{2} \right) v_{\hat{\pi}}^{\prime \prime}(x) + \mu v_{\hat{\pi}}^{\prime}(x) - q v_{\hat{\pi}} (x) + \sup_{0 \leq u \leq Kx} \left[ u \left(1-v_{\hat{\pi}}^\prime(x) \right) \right] = 0 .
\end{equation}
In this case, $\hat{\pi}$ is an optimal strategy for the control problem.
\end{lemma}

This lemma can be proved using standard arguments. The details are left to the reader.

We now provide a proof for Theorem~\ref{T:main}, the solution to the control problem. First, note that the Hamilton-Jacobi-Bellman (HJB) equation~\eqref{eq:HJB} is equivalent to
$$
\begin{cases}
\left( \frac{\sigma^2}{2} \right) v_{\hat{\pi}}^{\prime \prime}(x) + \mu v_{\hat{\pi}}^{\prime}(x) - q v_{\hat{\pi}} (x) = 0 & \text{if $v_{\hat{\pi}}^\prime(x) \geq 1$,}\\
\left( \frac{\sigma^2}{2} \right) v_{\hat{\pi}}^{\prime \prime}(x) + \mu v_{\hat{\pi}}^{\prime}(x) - q v_{\hat{\pi}} (x) + Kx \left(1-v_{\hat{\pi}}^\prime(x) \right) = 0 & \text{if $v_{\hat{\pi}}^\prime(x)<1$.}
\end{cases}
$$

Using the expression for $v_b$, obtained in Proposition~\ref{P:value-linear}, with the background material provided in Appendix~\ref{an:odes}, we have, for $0<x <b$,
$$
\left( \frac{\sigma^2}{2} \right) v_b^{\prime \prime}(x) + \mu v_b^{\prime}(x) - q v_b (x) = 0
$$
and, for $x>b$,
$$
\left( \frac{\sigma^2}{2} \right) v_b^{\prime \prime}(x) + \left(\mu - Kx \right) v_b^{\prime}(x) - q v_b (x) + Kx = 0 .
$$
Consequently, $v_{b^\ast}$ satisfies the HJB equation~\eqref{eq:HJB} if and only if
$$
\begin{cases}
v_{b^\ast}^\prime(x) \geq 1 & \text{for $0<x \leq b^\ast$,}\\
v_{b^\ast}^\prime(x) \leq 1 & \text{for $x > b^\ast$.}
\end{cases}
$$

Since $D_0=-(\mu/K)/\Hq(0)$, we can easily verify that $v_0^\prime (x) \leq 1$, for all $x>0$, if and only if
$$
- \frac{q^2}{\mu K} \Hq(0) \leq \Hqprime (x) .
$$
Recall that $\Hq$ is convex, so $\Hqprime$ is increasing. Thus, if $\mu K/q^2 \leq \Delta = - \Hq(0)/\Hqprime (0)$, then $\pi_0$ is optimal.

If $\mu K/q^2 > \Delta$, then, by Proposition~\ref{P:root}, the optimal level $b^\ast>0$ is given by~\eqref{eq:eqn-optimal-barrier} and thus we can show that
$$
C_{b^\ast} = \frac{(q+K)/K}{W^{(q) \prime} (b^\ast)} \qquad \text{and} \qquad D_{b^\ast} = \frac{q/K}{\Hqprime (b^\ast)} .
$$
Hence, by Proposition~\ref{P:value-linear}, the HJB equation is equivalent to
$$
\begin{cases}
W^{(q) \prime}(x) \geq \frac{(q+K)/K}{C_{b^\ast}}=W^{(q) \prime}(b^\ast) & \text{for $0<x \leq b^\ast$,}\\
\Hqprime (x) \geq \frac{q/K}{D_{b^\ast}}=\Hqprime (b^\ast) & \text{for $x > b^\ast$.}
\end{cases}
$$

The two inequalities are verified because $W^{(q) \prime}$ decreases on $(0,c^\ast)$ and $b^\ast<c^\ast$, and because $\Hq$ is a convex function.




\section{Discussion and numerical illustrations}\label{S:discussion}

For practical reasons such as solvency purposes, the company and shareholders might prefer to have $\mu > K b$, so that for some time, i.e., from the up-crossing of level $b$ until the process reaches level $\mu/K$, dividend payments do not cancel out all of the capital's growth. Indeed, we see that, below level $b$, the process $U^b$ behaves like a Brownian motion with constant (positive) drift $\mu$, and above level $b$, it behaves like an Ornstein-Uhlenbeck process with mean-reverting level $\mu/K$.

Therefore, let us now investigate the relationship between levels $b^\ast$ and $\mu/K$. First, it is known that $c^\ast < \mu/q$; see, e.g., \cite{gerber-shiu_2004}. If $K$ is \textit{relatively small}, i.e., if
$$
c^\ast < \frac{\mu}{q} \wedge \frac{\mu}{K} ,
$$
then the optimal barrier level $b^\ast$ is less than $\mu/K$. If $K$ is \textit{relatively large}, i.e., if
$$
\frac{\mu}{K} < c^\ast < \frac{\mu}{q} ,
$$
then, from Equation~\eqref{eq:optimal-barrier-2}, we deduce that $b^\ast < \mu/K$ if and only if
$$
\left(K+q \right) \left[ \frac{W^{(q)} (\mu/K)}{W^{(q) \prime} (\mu/K)} - \frac{\mu}{q} \right] > q \frac{\Hq (\mu/K)}{\Hqprime (\mu/K)} .
$$

\begin{remark}
Note that, if $\mu=0.3$ $\sigma = 4.5$, $q=0.05$ and $K=0.35$, then we do have that $\frac{\mu}{K} < b^\ast < c^\ast < \frac{\mu}{q}$.
\end{remark}


Now, let us further illustrate our main results. First, in Figure~\ref{fig:line}, we draw the value function for a delayed linear control strategy, as given in Proposition~\ref{P:value-linear}, as a function of the barrier level $b$, for two sets of parameters. We see that the optimal barrier level, given by the solution to Equation~\eqref{eq:eqn-optimal-barrier}, does correspond with the maximum of the function, in both cases. In the top panel, the parameters are such that $\mu K/q^2 > \Delta$, which is the condition in Proposition~\ref{P:root}, for the existence (and uniqueness) of a positive optimal barrier level $b^\ast>0$. In the bottom panel, parameters are such that $\mu K/q^2 < \Delta$ and we observe that the maximum is indeed attained at zero.

\begin{figure}[htbp]
\hbox{\hspace{-2cm}\includegraphics[scale=0.4]{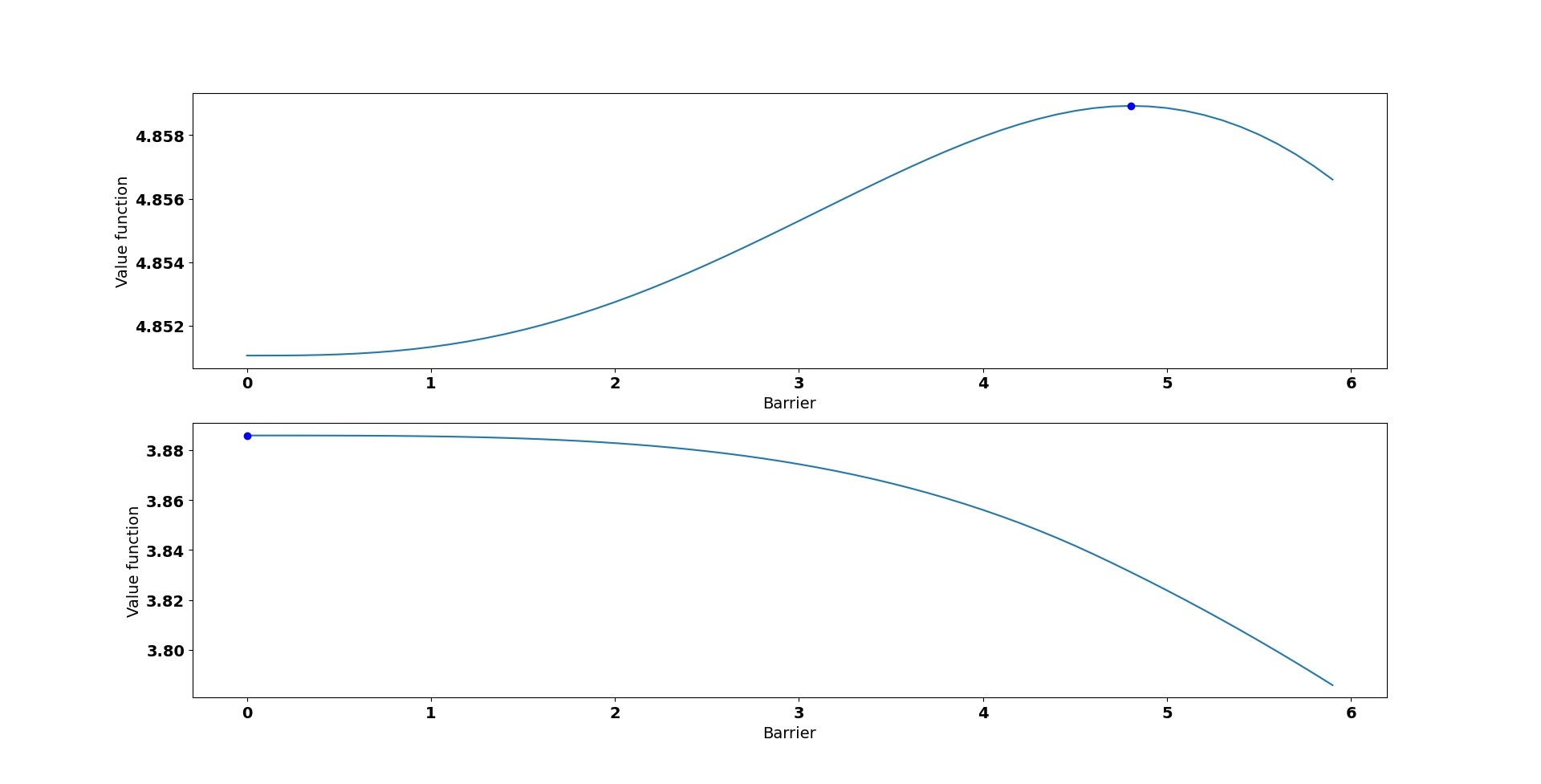}}
\caption{Value function of $\pi_b$ as a function of $b$. The dot indicates the value function at $b^\ast$ given by Equation~\eqref{eq:eqn-optimal-barrier}. Top panel: $\mu = 0.3$, $\sigma=4.5$, $K=0.1$, $q=0.025$ and $U_0 = 4.60$. Bottom panel: $\mu=0.3$, $\sigma=4.5$, $K=0.1$, $q=0.05$ and $U_0 = 4.60$.}\label{fig:line}
\end{figure}

Second, in Figure~\ref{fig:surface}, we draw the value function for a (delayed) linear control strategy as a function of two variables: the barrier level $b$ and the parameter $K$. The curve on the surface identifies the value function corresonding to the optimal barrier $b^\ast$. For small values of $K$, $\mu K/q^2 < \Delta$ and we see that the optimal barrier is $b^\ast=0$ while for values of $K$ such that $\mu K/q^2 > \Delta$ we see that $b^\ast>0$.  

\begin{figure}[htbp]
\hbox{\hspace{-4.5cm}\includegraphics[scale=0.47]{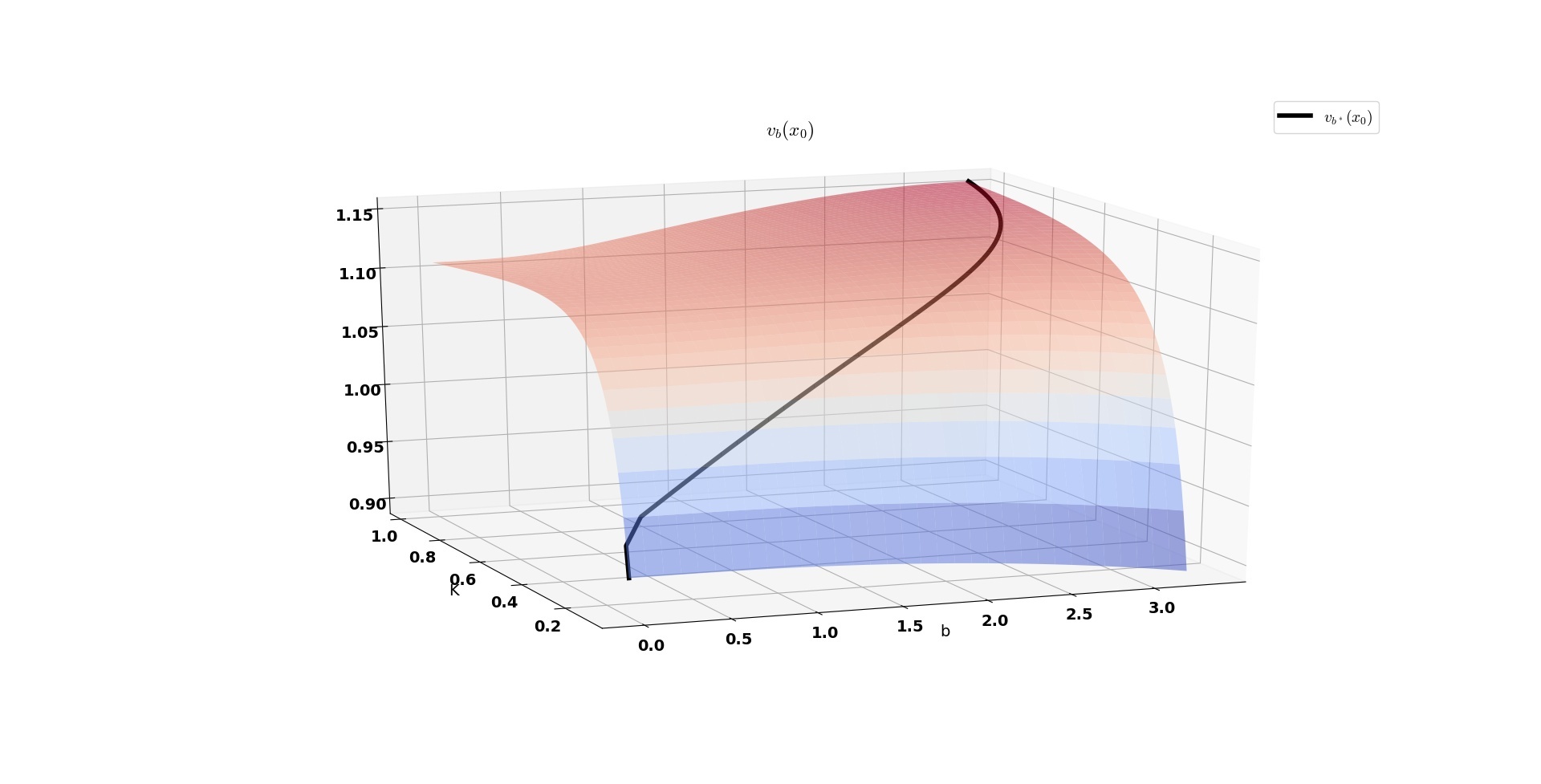}}
\caption{Value function of $\pi_b$ as a function of the variables $b$ and $K$. The curve is the value function of $\pi_{b^\ast}$. Parameters: $\mu=0.3$, $\sigma=2.5$, $q=0.07$ and $x_0 = 1$.}\label{fig:surface}
\end{figure}

Finally, as alluded to in the introduction, when $K$ goes to infinity, one expects to recover de Finetti's classical control problem, in which the optimal strategy is to pay out all surplus in excess of the barrier level $c^\ast$, as given by~\eqref{eq:classical-optimal-level}. Recall from Proposition~\ref{P:root} that the optimal level $b^\ast$ in our problem is always less than $c^\ast$. In Figure~\ref{fig:bstar_K}, we draw the value of the optimal barrier level $b^\ast$ as a function of $K$, i.e., $K \mapsto b^\ast (K)$. One can see that, for this set of parameters, the optimal barrier level $b^\ast(K)$ increases to $c^\ast$ as $K$ increases to infinity.

\begin{figure}[htbp]
\hbox{\hspace{-1.5cm}\includegraphics[scale=0.37]{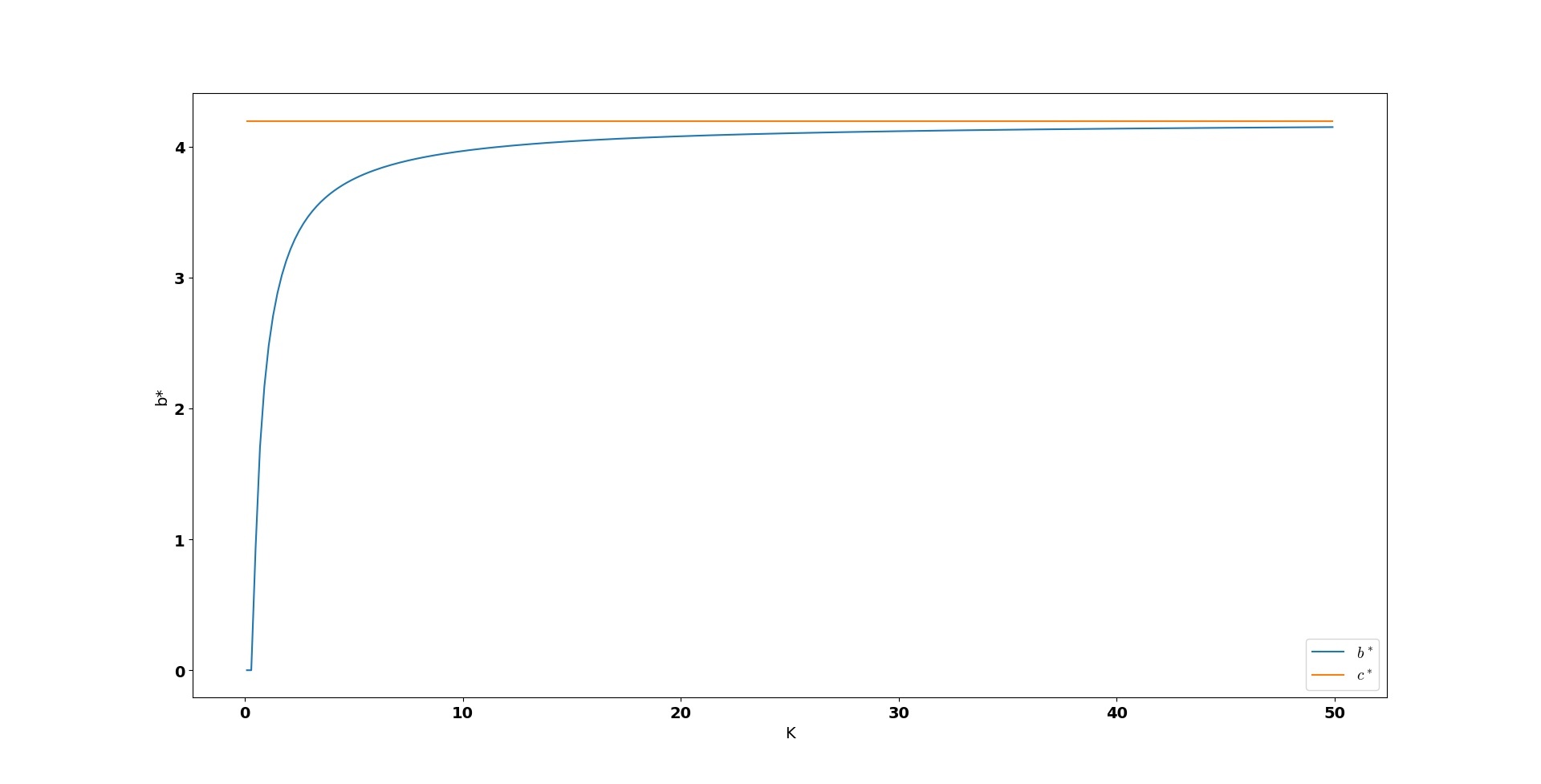}}
\caption{Optimal barrier level $b^\ast$ as a function of $K$. The red line corresponds to $c^\ast$. Parameters: $\mu=0.3$, $\sigma=4.5$, $q=0.07$ and $U_0 = 1$.}\label{fig:bstar_K}
\end{figure}

%

%
%
\bibliographystyle{abbrv}
\bibliography{references_affine-strategies}
%
%

\appendix

\section{Existence of the refracted diffusion process given by~\eqref{eq:main-sde}}\label{an:existence_proof}

It is well known, thanks to Proposition 3.6 in Section 5.3 of \cite{karatzas-shreve_1991}, that Equation~\eqref{eq:main-sde} has a weak solution. However, we can prove the existence of a strong solution by mimicking the proof of Lemma 12 in \cite{Kyprianou/Loeffen_2010}.

Fix an arbitrary $T>0$. For each $n\geq 1$, let $\{T^n_i \colon i=0,\dots, n\}$ be a sequence of partitions such that $T^n_0 = 0 <T^n_1<\cdots <T^n_n = T$ and $\lim_{n\to \infty} \max_{1\leq i \leq n} |T^n_i - T^n_{i-1}|=0$. For $t \in [0,T]$, let $X_t = \mu t + \sigma W_t$ and, for each $n$, define $X^n_t = \mu t + \sigma W_{T^n_i}$, if $T^n_i \leq t < T^n_{i+1}$, for $i=0,1,\dots, n-1$. One can show that the sequence of processes $\{X^n, n=1,2,\dots\}$ converges strongly to $X$, i.e.,
\[
\lim_{n \to \infty} \sup_{t \in [0,T]}|X^n_t - X_t| = 0
\]
almost surely. Now, define the sequence of processes $\{U^n, n=1,2,\dots\}$ by: for each $t \in [0,T]$, set
\[
U^n_t = X^n_t - K \int_0^t U^n_s \ind_{\{ U^n_s>b \}} \mathrm ds .
\]
Clearly, for each $n \geq 1$, $U^n$ is a well-defined process. Let us show that $\{U^n, n=1,2,\dots\}$ is a strong Cauchy sequence and thus converges to a process $U$ such that
\[
U_t = X_t - K \int_0^t U_s \ind_{\{ U_s>b \}} \mathrm ds .
\]
Define, for all for $t \in [0,T]$, $\Delta^{n,m} X_t := X^n_t - X^m_t$, $\Delta^{n,m} U_t := U^n_t - U^m_t$ and
$$
A^{n,m}_t := \Delta^{n,m} U_t - \Delta^{n,m} X_t = - K \int_0^t \left( U^n_s \ind_{\{ U^n_s>b \}} - U^m_s \ind_{\{ U^m_s>b \}} \right) \mathrm ds .
$$

Fix $\epsilon >0$. As $\{X^n, n=1,2\dots\}$ converges strongly to $X$, there exists an integer $N_\epsilon$ such that: if $n,m > N_\epsilon$, then $\sup_{t \in [0,T]}|\Delta^{n,m} X_t|<\epsilon$. Let us show by contradiction that, for any $n,m > N_\epsilon$, we have $\sup_{t \in [0,T]} |A^{n,m}_t| < \epsilon$. Assume it is not verified. Then, from the continuity of $t \mapsto A^{n,m}_t$ and since $A^{n,m}_0=0$, there exists $s \in [0,T]$ such that $|A^{n,m}_s| = \epsilon$ and such that, for any sufficiently small $\delta >0$, there exists $r_s \in [s,s + \delta)$ such that $|A^{n,m}_{r_s}| > \epsilon$. Assume that $A^{n,m}_s = \epsilon < A^{n,m}_{r_s}$.

Since, for all $t \in [0,T]$, $\Delta^{n,m} X_t \in (-\epsilon, \epsilon)$ and $A^{n,m}_t = \Delta^{n,m}U_t - \Delta^{n,m}X_t$, then $\Delta^{n,m}U_s>0$ and there exists $\delta>0$ such that $\Delta^{n,m}U_r>0$ for all $r \in [s,s+\delta)$. Consequently, for each $r \in [s,s+\delta)$,
\[
A^{n,m}_r - A^{n,m}_s = - K \int_s^r \left( U^n_v \ind_{\{ U^m_v>b \}} - U^m_v \ind_{\{ U^m_v>b \}} \right) \mathrm dv  < 0 .
\]
This is a contradiction. A similar argument can be used for the case $A^{n,m}_s = -\epsilon$.

In conclusion, for any $n,m > N_\epsilon$, we have $\sup_{t \in [0,T]} |A^{n,m}_t| < \epsilon$ and thus, by the triangle inequality, we have $\sup_{t \in [0,T]} |\Delta^{n,m} U_t| < 2\epsilon$.

\section{Differential equations and analytical properties}\label{an:odes}


It is well known (and easy to verify) that $W^{(q)}$, defined in~\eqref{eq:q-scale-function}, is a solution to the following ordinary differential equation (ODE):
$$
\frac{\sigma^2}{2} f^{\prime \prime}(x) + \mu f^{\prime}(x) - q f (x) = 0 , \quad x > 0 .
$$
Hence, for $x>0$, we have
$$
W^{(q) \prime \prime}(x) = \frac{2}{\sigma^2} \left( q W^{(q)} (x) - \mu W^{(q) \prime} (x) \right) .
$$

We are also interested in the following non-homogeneous ODE:
\begin{equation}\label{eq:non-homogeneous}
\frac{\sigma^2}{2} f^{\prime \prime}(x) + \left(\mu-Kx \right) f^{\prime}(x) - q f (x) + Kx = 0 , \quad x > 0 .
\end{equation}
We are looking for a solution of the form $f(x)=f_p(x)+f_h(x)$, where $f_p$ is a particular solution and where $f_h$ is a solution of the homogeneous version of~\eqref{eq:non-homogeneous}, that is
\begin{equation}\label{eq:homogeneous}
\frac{\sigma^2}{2} f^{\prime \prime}(x) + \left(\mu-Kx \right) f^{\prime}(x) - q f (x) = 0 , \quad x > 0 .
\end{equation}
On one hand, it is easy to verify that
$$
f_p(x) = \frac{K}{q+K} \left( x + \frac{\mu}{q} \right)
$$
is a solution to the ODE in~\eqref{eq:non-homogeneous}. On the other hand, from \cite{borodin-salminen_2002}, it is known that, for $\lambda>0$,
$$
b^2 f^{\prime \prime}(x) - x f^{\prime}(x) - \lambda f (x) = 0 , \quad x \in \reals ,
$$
admits
$$
\psi_\lambda (x) = \mathrm{e}^{x^2/4b^2} \mathrm{D}_{-\lambda} \left( -x/b \right) \quad \text{and} \quad \varphi_\lambda (x) = \mathrm{e}^{x^2/4b^2} \mathrm{D}_{-\lambda} \left( x/b \right) ,
$$
as its increasing solution and its decreasing solution, respectively, where $\mathrm{D}_{-\lambda}$ is the parabolic cylinder function defined by: for $x \in \reals$,
$$
\mathrm{D}_{-\lambda}(x) = \frac{1}{\Gamma(\lambda)} \mathrm{e}^{-x^2/4} \int_0^\infty t^{\lambda-1} \mathrm{e}^{-xt-t^2/2} \mathrm{d}t .
$$

Since $H_K^{(q)}(x)=\varphi_{q/K} \left( \left(x-\mu/K\right) \sqrt{2K} \right)$, as defined in~\eqref{eq:H-scale-function}, we deduce that $H_K^{(q)}$ is a solution to the homogeneous ODE in~\eqref{eq:homogeneous}. Hence, for $x>0$, we have
$$
H_K^{(q) \prime \prime}(x) = \frac{2}{\sigma^2} \left( \left(Kx - \mu \right) H_K^{(q) \prime}(x) + q H_K^{(q)}(x) \right) .
$$
Finally, we have that
$$
f(x) = \frac{K}{q+K} \left( x + \frac{\mu}{q} \right) + D_b H_K^{(q)}(x)
$$
is a solution to the non-homogeneous ODE in~\eqref{eq:non-homogeneous}.

\end{document}